\documentclass[11pt]{article}
\usepackage[utf8]{inputenc}
\usepackage[T1]{fontenc}
\usepackage{amsmath,amssymb,amsthm,amsfonts}
\usepackage{tikz}
\usepackage{graphicx}
\usepackage{subcaption}
\usepackage{mathrsfs}
\usepackage{enumitem}
\usepackage{geometry}
\usepackage{authblk}
\numberwithin{equation}{section}

\theoremstyle{plain}
\newtheorem{theorem}{Theorem}[section]

\newtheorem{lemma}[theorem]{Lemma}
\newtheorem{corollary}[theorem]{Corollary}
\newtheorem{proposition}[theorem]{Proposition}
\theoremstyle{definition}
\newtheorem{definition}[theorem]{Definition}
\newtheorem{remark}{Remark}

\numberwithin{equation}{section}

\newcommand{\T}{\mathscr T_K}
\newcommand{\M}{\mathrm M_\infty(K)}
\newcommand{\tr}{\operatorname{Tr}}
\newcommand{\Span}{\operatorname{span}_K}
\newcommand{\rad}{\operatorname{rad}}
\newcommand{\ch}{\operatorname{char}}

\title{Lie Ideals and Lie Derivations of the Algebraic Toeplitz Algebra}
\author[1,2]{Nguyen Huynh Thao Nhi%
\thanks{Email: \texttt{nhinht\_ph@utc.edu.vn}}}
\author[1]{Huynh Viet Khanh%
\thanks{Email: \texttt{khanhhv@hcmue.edu.vn}}}

\affil[1]{Department of Mathematics and Informatics,
Ho Chi Minh City University of Education,
Ho Chi Minh City, Vietnam}

\affil[2]{Campus in Ho Chi Minh City,
University of Transport and Communications,
Ho Chi Minh City, Vietnam}
\date{}

\begin{document}
\maketitle
\begin{abstract}
Let $K$ be a field and let $\T=L_K(E_T)$ be the algebraic Toeplitz algebra. We classify the Lie ideals of $\T^-$. If $w$ is the sink of the Toeplitz graph and $F=I(w)$, then $F\cong M_\infty(K)$, $\T/F\cong K[t,t^{-1}]$, and $[\T,\T]=F$. We prove that every noncentral Lie ideal contains $\mathfrak s=[F,F]$, the finitary trace-zero Lie algebra. Thus the classification reduces to the Heisenberg-type quotient $\T/\mathfrak s$. The Lie ideals are precisely $0$, $K1$, the spaces $\mathfrak s+U$, where $U\leq P_K\oplus Kw$ and $w\notin U$, and the inverse images $\rho^{-1}(W)$, where $W\leq K[t,t^{-1}]$. Here $P_K=K1$ in characteristic $0$, and $P_K=\operatorname{span}_K\{1,e^{m\ell},(e^*)^{m\ell}:m\geq1\}$ if $\operatorname{char}K=\ell>0$. We also describe the Lie derivations of $\T^-$ as sums of associative derivations of $\T$ and central Lie derivations factoring through $\T/F$.
\end{abstract}

\section{Introduction}

Let $K$ be a field. If $A$ is an associative $K$-algebra, then $A^-$ denotes the Lie algebra with the same underlying vector space and bracket $[a,b]=ab-ba$. Every two-sided ideal of $A$ is a Lie ideal of $A^-$, but the converse need not hold. For Leavitt path algebras this distinction is important: their two-sided ideals are well understood, while their Lie ideals are usually harder to describe.

Lie structures of Leavitt path algebras have been studied from several points of view. Mesyan described in \cite{Mesyan2013} the commutator subspace $[L_K(E),L_K(E)]$ and classified the Leavitt path algebras satisfying $L_K(E)=[L_K(E),L_K(E)]$; in this commutator case every Lie ideal is a two-sided ideal. Khanh classified the Leavitt path algebras for which every Lie ideal is a two-sided ideal \cite{Huynh2024}. The present paper treats the complementary case of the algebraic Toeplitz algebra: its Lie ideal lattice is much larger than its associative ideal lattice, but it is still explicit.

Let $\T=L_K(E_T)$ be the Leavitt path algebra of the Toeplitz graph $E_T$:
\begin{center}
		\begin{tikzpicture}
			\node at (0,0) (0) {$\bullet$};
			\node at (0,-0.3) {$v$};
			\node at (1.5,0) (2) {$\bullet$};
			\node at (1.5,-0.3) {$w$};
			
			\draw [->] (0) to [out=225,in=135,looseness=6] node[left] {$e$} (0);
			\draw [->] (0) to node[above] {$f$} (2);
		\end{tikzpicture}
	\end{center}
This is the algebraic Toeplitz algebra, and it is isomorphic to the Jacobson algebra $K\langle X,Y\mid XY=1\rangle$ \cite[Proposition~1.3.7]{AbramsAraSiles}. Gerritzen studied in \cite{Gerritzen2000} the Jacobson algebra in detail, including its right ideals, module-theoretic structure, associative derivations, and automorphisms. We use the finite-rank ideal and Laurent polynomial quotient from this picture, but work throughout in the path language of $E_T$. Thus the sink $w$ generates the finitary matrix ideal $F=I(w)\cong M_\infty(K)$, and $\T/F\cong K[t,t^{-1}]$.

The algebra $\T$ has the natural $\mathbb Z$-grading inherited from the Leavitt path algebra grading: vertices have degree $0$, real edges degree $1$, and ghost edges degree $-1$. For the Toeplitz graph, $\deg e=\deg f=1$, $\deg e^*=\deg f^*=-1$, and $\deg w=0$. Both $F$ and $\mathfrak s=[F,F]$ are graded Lie ideals. Once the full Lie ideal lattice is known, the graded Lie ideals are obtained by imposing the same grading condition on the corresponding subspaces of $K[t,t^{-1}]$ and $P_K\oplus Kw$. This provides a useful contrast with the graded associative ideal lattice, which consists only of $0$, $F$, and $\T$.

The main result is the classification of the Lie ideals of $\T^-$. Since
$\T/F\cong K[t,t^{-1}]$ is abelian, the inverse image of every
$K$-subspace of $K[t,t^{-1}]$ is a Lie ideal of $\T^-$ containing
$F$. Most of these inverse images are not associative ideals. A second
family consists of Lie ideals containing $\mathfrak s=[F,F]$ but not
$F$; its form depends on $\operatorname{char}K$. Every member of this
family meets $F$ in $\mathfrak s$ and, unless it equals $\mathfrak s$,
is incomparable with $F$.

We first show that $[\T,\T]=F$. Then a finitary matrix-unit extraction argument, in the spirit of Ho\l{}ubowski and \.{Z}urek \cite[Section~2]{HolubowskiZurek}, shows that every noncentral Lie ideal of $\T^-$ contains $\mathfrak s$. Hence the problem reduces to $\T/\mathfrak s$. This quotient is a Heisenberg-type central extension of the abelian vector space $K[t,t^{-1}]$. Its center determines the Lie ideals containing $\mathfrak s$ whose
images in $\T/\mathfrak s$ are central, and this gives the complete
classification.

We also record the Lie derivations of $\T^-$. The associative derivations of the Jacobson algebra were described by Gerritzen in \cite{Gerritzen2000}, and Lopatkin later described in \cite{Lopatkin2019} derivations and outer derivations of Leavitt path algebras, including the Lie algebra structure of the outer derivation algebra of the Toeplitz algebra. We recall only the form needed here. The additional point is that Lie derivations of $\T^-$ include central Lie derivations factoring through $\T/F$, because $[\T,\T]=F\neq\T$.

The paper is organized as follows. Section~2 fixes the notation for Leavitt path algebras and infinite matrices. Section~3 records the path basis and the Laurent quotient for the Toeplitz algebra. Section~4 proves the classification of Lie ideals and describes the graded sublattice. Section~5 records the corresponding Lie derivations.

\section{Preliminaries and notation}

Let $K$ be a field. We recall first the basic notation for Leavitt path algebras. A directed graph $E=(E^0,E^1,r,s)$ consists of a set $E^0$ of vertices, a set $E^1$ of edges, and maps $s,r:E^1\to E^0$, called the source and range maps. A vertex $v$ is a \textit{sink} if $s^{-1}(v)=\emptyset$, and a \textit{source} if $r^{-1}(v)=\emptyset$. A vertex $v$ is \textit{regular} if $0<|s^{-1}(v)|<\infty$, and it is an \textit{infinite emitter} if $|s^{-1}(v)|=\infty$. A path is either a vertex, viewed as a path of length $0$, or a finite sequence $p=e_1\cdots e_n$ of edges such that $r(e_i)=s(e_{i+1})$ for $1\leq i<n$. We write $s(p)=s(e_1)$, $r(p)=r(e_n)$, and $|p|=n$ for the source, the range, and the length of $p$ respectively.

\begin{definition}[Leavitt path algebras]
Let $E$ be an arbitrary directed graph and $K$ any field. We define a set $(E^1)^*$ consisting of symbols of the form $\{e^*\mid e\in E^1\}$. The \emph{Leavitt path algebra} of $E$ with coefficients in $K$, denoted $L_K(E)$, is the free associative $K$-algebra generated by the set $E^0\cup E^1\cup (E^1)^*$, subject to the following relations:
\begin{enumerate}
\item[\rm(V)] $vv'=\delta_{v,v'}v$ for all $v,v'\in E^0$;
\item[\rm(E1)] $s(e)e=er(e)=e$ for all $e\in E^1$;
\item[\rm(E2)] $r(e)e^*=e^*s(e)=e^*$ for all $e\in E^1$;
\item[\rm(CK1)] $e^*e'=\delta_{e,e'}r(e)$ for all $e,e'\in E^1$;
\item[\rm(CK2)] $v=\sum_{{e\in E^1\mid s(e)=v}}ee^*$ for every regular vertex $v\in E^0$.
\end{enumerate}
\end{definition}

We fix the following notation for infinite matrices. Let $\mathbb N=\{1,2,\ldots\}$, and let $V$ be a $K$-vector space with basis $\mathscr B=\{e_1,e_2,\ldots\}$. Here $\mathcal L(V)=\operatorname{End}_K(V)$ denotes the unital $K$-algebra of all $K$-linear endomorphisms of $V$. Once $\mathscr B$ is fixed, every $T\in\mathcal L(V)$ is represented by a unique column-finite matrix $(a_{ij})$, where $T(e_j)=\sum_i a_{ij}e_i$. Thus $\mathcal L(V)$ is identified with the unital ring $\mathrm M_{cf}(K)$ of column-finite $\mathbb N\times\mathbb N$ matrices over $K$. Let $\mathrm M_{rcf}(K)$ be the subring of $\mathrm M_{cf}(K)$ consisting of matrices whose rows and columns are both finite, and let $\mathrm M_\infty(K)$ be the ring of finitary matrices. Then $\mathrm M_\infty(K)$ is a two-sided ideal of $\mathrm M_{rcf}(K)$.

We denote by $E_{ij}$ the standard matrix unit whose $(i,j)$-entry is $1$ and whose remaining entries are $0$. Thus $E_{ij}e_j=e_i$, $E_{ij}e_k=0$ if $k\neq j$, and $E_{ij}E_{kl}=\delta_{jk}E_{il}$. The off-diagonal matrix units are the $E_{ij}$ with $i\neq j$; the diagonal matrix units are the idempotents $E_{ii}$; and a diagonal difference is an element $E_{ii}-E_{jj}$. For $A=(a_{ij})\in\mathrm M_\infty(K)$, put $\operatorname{tr}(A)=\sum_{i\geq 1}a_{ii}$; this sum is finite since $A$ has only finitely many nonzero entries. Hence
$$
\begin{aligned}
[\mathrm M_\infty(K),\mathrm M_\infty(K)]
&=\{A\in\mathrm M_\infty(K):\operatorname{Tr}(A)=0\}\\
&=\operatorname{span}_K\{E_{ij}:i\neq j\}
+\operatorname{span}_K\{E_{ii}-E_{jj}:i,j\in\mathbb N\}.
\end{aligned}
$$
In particular, $\mathrm M_\infty(K)=KE_{11}\oplus[\mathrm M_\infty(K),\mathrm M_\infty(K)]$. We reserve $E_{ij}$ for the standard matrix units in $\mathrm M_\infty(K)$; the symbols $F_{ij}$ will denote their concrete realization inside $\T$.

\section{The Toeplitz algebra}

Let $E_T$ be the Toeplitz graph and put $\T=L_K(E_T)$. Thus $1=v+w$, $e^*e=v$, $f^*f=w$, $e^*f=f^*e=0$, and $v=ee^*+ff^*$. We will work directly with paths in $E_T$. Let
$$
        \alpha_1=w\qquad\text{and}\qquad \alpha_i=e^{i-2}f\quad\text{where }i\geq2.
$$
Then $\{\alpha_i:i\geq1\}$ is precisely the set of paths in $E_T$ ending at the sink $w$. For $i,j\geq1$, define $F_{ij}=\alpha_i\alpha_j^*$.

\begin{lemma}\label{lem:matrix-units}
Let $I(w)$ be the ideal of $\T$ genenerated by $w$. For all $i,j,k,l\geq1$, $F_{ij}F_{kl}=\delta_{jk}F_{il}$. Consequently $F:=\bigoplus_{i,j\geq1}KF_{ij}$ is a two-sided ideal of $\T$, $F\cong\M$, and $F=I(w)$.
\end{lemma}

\begin{proof}
Since $w$ is a sink, the paths $\alpha_i$ are mutually incomparable. Hence $\alpha_j^*\alpha_k=0$ if $j\neq k$, while $\alpha_j^*\alpha_j=w$. Therefore $F_{ij}F_{kl}=\alpha_i(\alpha_j^*\alpha_k)\alpha_l^*=\delta_{jk}\alpha_i\alpha_l^*=\delta_{jk}F_{il}$. Thus the $F_{ij}$'s are matrix units. The ideal $I(w)$ is spanned by all elements $\alpha\beta^*$ with $r(\alpha)=r(\beta)=w$, and these paths are precisely the $\alpha_i$'s. Hence $F=I(w)$.
\end{proof}

We shall use the standard path basis of a Leavitt path algebra:

\begin{lemma}\label{lem:path-basis}
The set
$$
        \mathcal B=\{v\}\cup\{e^n:n\geq1\}\cup\{(e^*)^n:n\geq1\}\cup\{F_{ij}:i,j\geq1\}
$$
is a $K$-basis of $\T$. Consequently, as a vector space, we have
$$
        \T=K1\oplus\Bigl(\bigoplus_{n\geq1}Ke^n\Bigr)\oplus\Bigl(\bigoplus_{n\geq1}K(e^*)^n\Bigr)\oplus F.
$$
\end{lemma}

\begin{proof}
This is the standard path basis theorem for Leavitt path algebras given in \cite[Corollary~1.5.12]{AbramsAraSiles}, applied with $e$ as the distinguished edge at $v$. Indeed, the paths ending at $v$ are $v,e,e^2,\ldots$, and the corollary excludes the monomials $e^m(e^*)^n$ with both $m,n\geq1$. The paths ending at $w$ are the $\alpha_i$'s, and they give exactly the matrix units $F_{ij}=\alpha_i\alpha_j^*$. Since $1=v+w$ and $w=F_{11}\in F$, the stated direct sum decomposition follows.
\end{proof}

The following lemma is well-known; so we ommit the proof:

\begin{lemma}\label{lem:quotient}
Let $\rho:\T\to K[t,t^{-1}]$ be the homomorphism determined by $\rho(v)=1$, $\rho(e)=t$, $\rho(e^*)=t^{-1}$, and $\rho(w)=\rho(f)=\rho(f^*)=0$. Then $\ker\rho=F$, and hence $\T/F\cong K[t,t^{-1}]$.
\end{lemma}

\section{The Lie ideals}
Let $\mathfrak s=[F,F]$. Under the identification $F\cong\mathrm M_\infty(K)$, the elements of $\mathfrak s$ are precisely those matrices having trace zero; thus $\mathfrak s=\{a\in F\mid \operatorname{tr}(a)=0\}$.

\begin{lemma}\label{lem:s-trace-zero}
One has $F=KF_{11}\oplus\mathfrak s$ as vector spaces.
\end{lemma}

\begin{proof}
Every commutator in the finitary matrix algebra has trace zero. Conversely, for $i\neq j$, $F_{ij}=[F_{ii},F_{ij}]$, and $F_{ii}-F_{jj}=[F_{ij},F_{ji}]$. Hence $[F,F]$ contains all off-diagonal matrix units and all diagonal differences, which span exactly $\mathfrak s$. Finally, if $A\in F$ and $\lambda=\tr(A)$, then $A-\lambda F_{11}\in\mathfrak s$, while $KF_{11}\cap\mathfrak s=0$ because $\tr(F_{11})=1$.
\end{proof}

\begin{proposition}\label{prop:derived}
One has $[\T,\T]=F$.
\end{proposition}

\begin{proof}
Since $\T/F$ is commutative, $[\T,\T]\subseteq F$. For the reverse inclusion, write
$$
[e^*,e]=e^* e-e e^*=v-e e^*=f f^*=F_{22},
\qquad
[f^*,f]=f^* f-f f^*=w-f f^*=F_{11}-F_{22}.
$$
Hence $F_{11}=w=[e^*,e]+[f^*,f]\in[\T,\T]$. Also $\mathfrak s=[F,F]\subseteq[\T,\T]$. Since $F=KF_{11}\oplus\mathfrak s$, it follows that $F\subseteq[\T,\T]$. Thus $[\T,\T]=F$.
\end{proof}

\begin{lemma}\label{lem:s-ideal}
The subspace $\mathfrak s$ is a Lie ideal of $\T^-$.
\end{lemma}

\begin{proof}
Since $F\triangleleft\T$, we have $[\T,F]\subseteq F$. Hence, by the Jacobi identity,
$$
[\T,\mathfrak s]=[\T,[F,F]]
\subseteq [[\T,F],F]+[F,[\T,F]]
\subseteq [F,F]=\mathfrak s.
$$
Thus $\mathfrak s\triangleleft\T^-$.
\end{proof}

The above proof uses a finitary analogue method of Ho\l{}ubowski and \u{Z}urek given in \cite[Section~2]{HolubowskiZurek}.

\begin{lemma}\label{lem:F-lie-ideals}
If $J$ is a nonzero Lie ideal of $F^-$, then $\mathfrak s\subseteq J$. Hence the Lie ideals of $F^-$ are precisely $0$, $\mathfrak s$, and $F$.
\end{lemma}

\begin{proof}
Choose $0\neq A=\sum_{(i,j)\in\Omega}a_{ij}F_{ij}\in J$, with $\Omega$ finite. Choose distinct $r,s\in\mathbb N$ which occur neither as a first coordinate nor as a second coordinate of any pair in $\Omega$. Pick $(i,j)\in\Omega$ with $a_{ij}\neq0$. Then
$$
[F_{ri},A]=F_{ri}A=\sum_k a_{ik}F_{rk}\in J.
$$
Hence
$$
\big[\sum_k a_{ik}F_{rk},F_{js}\big]=a_{ij}F_{rs}\in J,
$$
so $F_{rs}\in J$. Since $J$ is a Lie ideal of $F^-$,
$$
[F_{ar},F_{rs}]=F_{as}\quad\text{if }a\neq s,
\quad
\text{and}
\quad
[F_{rs},F_{sb}]=F_{rb}\quad\text{if }b\neq r.
$$
Thus $F_{as}\in J$ for all $a\neq s$. If $a\neq b$ and $a\neq s$, then $[F_{as},F_{sb}]=F_{ab}$, so $F_{ab}\in J$. If $b\neq s$, choose $c\notin\{s,b\}$. Then $F_{cb}\in J$, and $[F_{cb},F_{sc}]=-F_{sb}$. Hence every off-diagonal matrix unit belongs to $J$. Finally,
$$
[F_{ab},F_{ba}]=F_{aa}-F_{bb}\quad\text{if }a\neq b,
$$
so $J$ contains all diagonal differences. These elements span $\mathfrak s$, and therefore $\mathfrak s\subseteq J$.

Since the trace map induces $F/\mathfrak s\cong K$, any Lie ideal of $F^-$ containing $\mathfrak s$ is either $\mathfrak s$ or $F$. The assertion follows.
\end{proof}

\begin{lemma}\label{lem:center}
The center of $\T$ is $K1$.
\end{lemma}

\begin{proof}
Let $a\in Z(\T)$. By Lemma~\ref{lem:path-basis},
$$
a=\lambda+\sum_{m=1}^M\alpha_m e^m+\sum_{n=1}^N\beta_n(e^*)^n+f,
$$
where $f\in F$.
Write $f=\sum_{(i,j)\in\Omega}c_{ij}F_{ij}$ with $\Omega$ finite. Choose $r>M+1,N+1$ such that $r$ occurs neither as a first coordinate nor as a second coordinate of any pair in $\Omega$. Then $fF_{rr}=F_{rr}f=0$. Since $a$ commutes with $F_{rr}$,
$$
0=[a,F_{rr}]
=\sum_{m=1}^M\alpha_m(F_{r+m,r}-F_{r,r-m})
+\sum_{n=1}^N\beta_n(F_{r-n,r}-F_{r,r+n}).
$$
The displayed matrix units are distinct, so $\alpha_m=0$ and $\beta_n=0$ for all $m,n$. Hence $a=\lambda1+f$.

Now $f$ commutes with every $F_{ij}$. Fix $i$. Choose $j$ which occurs as no first coordinate of any pair in $\Omega$. Then $F_{ij}f=0$, and
$$
0=[f,F_{ij}]=fF_{ij}=\sum_p c_{pi}F_{pj}.
$$
Thus $c_{pi}=0$ for all $p$. Since $i$ was arbitrary, $f=0$. Therefore $a=\lambda1$, and $Z(\T)=K1$.
\end{proof}

\begin{proposition}\label{prop:noncentral-contains-s}
Let $L$ be a Lie ideal of $\T^-$. If $L\not\subseteq K1$, then $\mathfrak s\subseteq L$.
\end{proposition}

\begin{proof}
Choose $a\in L\setminus K1$. By Lemma~\ref{lem:center}, $a\notin Z(\T)$, so there is $b\in\T$ with $[a,b]\neq0$. Since $L$ is a Lie ideal of $\T^-$, $[a,b]\in L$. Since $\T/F$ is commutative, $[a,b]\in F$. Hence $L\cap F\neq0$.

Now $L\cap F$ is a Lie ideal of $F^-$. Indeed, if $x\in L\cap F$ and $y\in F$, then $[x,y]\in L$, because $L\triangleleft\T^-$, and $[x,y]\in F$, because $F$ is a (non-unital) subalgebra of $\T$. Thus $[x,y]\in L\cap F$. By Lemma~\ref{lem:F-lie-ideals}, we get $\mathfrak s\subseteq L\cap F\subseteq L$.
\end{proof}

Before we go further, let us make the following observations: Let $\pi:\T\to\T/\mathfrak s$ be the quotient map. Set
$c=\pi(w)$, $z_0=\pi(1)$, $z_n=\pi(e^n)$ for $n>0$, and
$z_{-n}=\pi((e^*)^n)$ for $n>0$. By
Lemmas~\ref{lem:path-basis} and~\ref{lem:s-trace-zero}, these elements
span $\T/\mathfrak s$. We claim that the sum is direct. Indeed, since $\mathfrak s\subseteq F=\ker\rho$, the map $\rho$ induces a
linear map
$$
\overline{\rho}:\T/\mathfrak s\longrightarrow K[t,t^{-1}]
$$
defined by
$$
\overline{\rho}(c)=0,\qquad
\overline{\rho}(z_0)=1,\qquad
\overline{\rho}(z_n)=t^n,\qquad
\overline{\rho}(z_{-n})=t^{-n}.
$$
Suppose that
$$
\alpha c+\lambda_0z_0
+\sum_{n=1}^{N}\lambda_nz_n
+\sum_{n=1}^{M}\mu_nz_{-n}=0.
$$
Applying $\overline{\rho}$ to this equality yields
$$
\lambda_0+\sum_{n=1}^{N}\lambda_nt^n
+\sum_{n=1}^{M}\mu_nt^{-n}=0.
$$
It follows that $\lambda_0=\lambda_n=\mu_n=0$; and so the original relation therefore reduces to $\alpha c=0$, whence $\alpha w\in\mathfrak s$. Under the identification $F\cong\mathrm M_\infty(K)$, the vertex $w$ corresponds to $E_{11}$, so $\operatorname{tr}(w)=1$. Since every element of $\mathfrak s$ has trace zero, we have $0=\operatorname{tr}(\alpha w)=\alpha\operatorname{tr}(w)=\alpha$. Thus all the coefficients vanish, and the displayed sum is direct.

\begin{lemma}\label{lem:bracket-quotient}
For all $m,n\in\mathbb Z$, one has $[z_m,z_n]=-m\delta_{m+n,0}c$ and $[c,z_n]=0$.
\end{lemma}

\begin{proof}
We have that $[w,e]=[w,e^*]=0$, while $[w,f]=-f=-F_{21}$ and $[w,f^*]=f^*=F_{12}$; also $[w,F]\subseteq[F,F]=\mathfrak s$. Hence $[c,z_n]=0$ for all $n$. It follows that $c$ is central.

We now compute the brackets involving positive and negative powers. Put $q=ff^*$. We first claim that, for every $a>0$,
$$
e^a(e^*)^a=v-\sum_{r=0}^{a-1}e^rq(e^*)^r.
$$
For $a=1$, this is the relation $ee^*=v-ff^*$. If the formula holds for $a$, then
$$
\begin{aligned}
e^{a+1}(e^*)^{a+1}
&=e\left(v-\sum_{r=0}^{a-1}e^rq(e^*)^r\right)e^* \\
&=ee^*-\sum_{r=0}^{a-1}e^{r+1}q(e^*)^{r+1} \\
&=v-\sum_{r=0}^{a}e^rq(e^*)^r,
\end{aligned}
$$
which proves the claim by induction.

Let $b\geq a>0$. Since $(e^*)^ae^b=e^{b-a}$, the preceding identity yields
$$
\begin{aligned}
[(e^*)^a,e^b]
&=e^{b-a}-e^{b-a}e^a(e^*)^a \\
&=\sum_{r=0}^{a-1}e^{b-a+r}q(e^*)^r \\
&=\sum_{r=0}^{a-1}F_{b-a+r+2,r+2}.
\end{aligned}
$$
If $b>a$, the two indices of every matrix unit in this sum are distinct. For $i\neq j$, we have $F_{ij}=[F_{ii},F_{ij}]\in\mathfrak s$, so $[(e^*)^a,e^b]\in\mathfrak s$ whenever $b>a$.

If $b=a$, the same calculation shows $[(e^*)^a,e^a]=\sum_{r=0}^{a-1}F_{r+2,r+2}$. For every $r\geq0$, one has $F_{r+2,r+2}-w =[F_{r+2,1},F_{1,r+2}]\in\mathfrak s$. Consequently, $[(e^*)^a,e^a]\equiv aw\pmod{\mathfrak s}$.

It remains to consider $a>b>0$. In this case $(e^*)^ae^b=(e^*)^{a-b}$, while $e^b(e^*)^a=e^b(e^*)^b(e^*)^{a-b}$. Applying the preceding formula with exponent $b$ we have
$$
\begin{aligned}
[(e^*)^a,e^b]
&=(e^*)^{a-b}-e^b(e^*)^b(e^*)^{a-b} \\
&=\sum_{r=0}^{b-1}e^rq(e^*)^{r+a-b} \\
&=\sum_{r=0}^{b-1}F_{r+2,r+a-b+2}.
\end{aligned}
$$
Since $a>b$, every matrix unit in this sum is off-diagonal and hence belongs to $\mathfrak s$.

Passing to $\T/\mathfrak s$, we have therefore proved that
$$
[z_{-a},z_b]
=
\begin{cases}
ac, & a=b,\\
0, & a\neq b,
\end{cases}
\qquad \text{where }a,b>0.
$$
Equivalently, we have $[z_{-a},z_b]=a\delta_{a,b}c$. By skew-symmetry, it follows that $[z_a,z_{-b}]=-a\delta_{a,b}c$. Brackets between two positive powers of $e$, or between two positive powers of $e^*$, vanish because powers of the same element commute. Moreover, $z_0=\pi(1)$ is central. Combining these cases yields
$$
[z_m,z_n]=-m\delta_{m+n,0}c,
\qquad\text{for all } m,n\in\mathbb Z.
$$
\end{proof}

Let $\mathcal V=\Span\{z_n:n\in\mathbb Z\}$. According to Lemma~\ref{lem:bracket-quotient}, we get
$$
\T/\mathfrak s=\mathcal V\oplus Kc,
\qquad [u,v]=\beta(u,v)c,\qquad [c,\T/\mathfrak s]=0,
$$
where $\beta:\mathcal V\times\mathcal V\to K$ is given by $\beta(z_m,z_n)=-m\delta_{m+n,0}$. Write $\rad(\beta)=\{u\in\mathcal V:\beta(u,v)=0\text{ for all }v\in\mathcal V\}$. If $u=\sum_m a_mz_m$, then $\beta(u,z_n)=na_{-n}$. Hence
$$
\rad(\beta)=Kz_0\quad\text{if }\operatorname{char}K=0,
\quad
\text{and}
\quad
\rad(\beta)=\Span\{z_{m\ell}:m\in\mathbb Z\}\quad\text{if }\operatorname{char}K=\ell>0.
$$
Consequently,
$$
Z(\T/\mathfrak s)=\rad(\beta)\oplus Kc.
$$

\begin{lemma}\label{lem:heisenberg-ideals}
Let $\mathfrak h=\mathcal V\oplus Kc$ be a Lie algebra such that $[u,v]=\beta(u,v)c$ and $[c,\mathfrak h]=0$. Then the Lie ideals of $\mathfrak h$ are precisely the subspaces contained in $Z(\mathfrak h)$ and the subspaces $Kc\oplus W$, where $W\leq\mathcal V$.
\end{lemma}

\begin{proof}
The listed subspaces are Lie ideals. Conversely, let $J\triangleleft\mathfrak h$. If $J\subseteq Z(\mathfrak h)$, there is nothing to prove. Otherwise choose $x=u+\lambda c\in J$ with $u\notin\rad(\beta)$. Then $\beta(u,v)\neq0$ for some $v\in\mathcal V$, and hence $[x,v]=\beta(u,v)c\in J$. Thus $c\in J$.

Let $W$ be the image of $J$ under the projection $\mathcal V\oplus Kc\to\mathcal V$. If $w\in W$, then $w+\alpha c\in J$ for some $\alpha\in K$; since $c\in J$, it follows that $w\in J$. Hence $Kc\oplus W\subseteq J$. The reverse inclusion follows from the definition of $W$. Therefore $J=Kc\oplus W$.
\end{proof}

Set
$$
P_K=
\begin{cases}
K1, & \text{if }\ch K=0,\\
\Span\{1,e^{m\ell},(e^*)^{m\ell}:m\geq1\}, & \text{if }\ch K=\ell>0.
\end{cases}
$$
Then $\pi(P_K)=\rad(\beta)$, and $\pi$ restricts to a vector-space isomorphism $P_K\oplus Kw\to Z(\T/\mathfrak s)$.

\begin{theorem}\label{thm:classification}
The Lie ideals of $\T^-$ are precisely the following:
\begin{enumerate}[label=\rm(\roman*)]
\item $0$ and $K1$;
\item $\mathfrak s+U$, where $U\leq P_K\oplus Kw$ and $w\notin U$;
\item $\rho^{-1}(W)$, where $W\leq K[t,t^{-1}]$ is a $K$-subspace.
\end{enumerate}
\end{theorem}

\begin{proof}
The listed subspaces are Lie ideals. The spaces $0$ and $K1$ are central. If $U\leq P_K\oplus Kw$, then $(\mathfrak s+U)/\mathfrak s\subseteq Z(\T/\mathfrak s)$, so $\mathfrak s+U$ is a Lie ideal. Finally, if $W\leq K[t,t^{-1}]$, then $\rho^{-1}(W)$ is a Lie ideal, because $\T/F$ is abelian and $F=[\T,\T]\subseteq\rho^{-1}(W)$.

Conversely, let $L\triangleleft\T^-$. If $L\subseteq K1$, then $L=0$ or $L=K1$. Otherwise Proposition~\ref{prop:noncentral-contains-s} gives $\mathfrak s\subseteq L$. Hence $L/\mathfrak s$ is a Lie ideal of $\T/\mathfrak s$. By Lemma~\ref{lem:heisenberg-ideals}, either $L/\mathfrak s\subseteq Z(\T/\mathfrak s)$, or $Kc\subseteq L/\mathfrak s$.

Assume first that $L/\mathfrak s\subseteq Z(\T/\mathfrak s)$. Since $P_K\oplus Kw\to Z(\T/\mathfrak s)$ is an isomorphism, $L=\mathfrak s+U$ for a unique subspace $U\leq P_K\oplus Kw$. If $w\notin U$, then $L$ is of type {\rm(ii)}. If $w\in U$, then $F=\mathfrak s\oplus Kw\subseteq L$, and hence $L=\rho^{-1}(\rho(L))$, so $L$ is of type {\rm(iii)}.

It remains to consider the case $Kc\subseteq L/\mathfrak s$. Since $c=w+\mathfrak s$ and $\mathfrak s\subseteq L$, we get $w\in L$. Thus $F=\mathfrak s\oplus Kw\subseteq L$. Therefore $L=\rho^{-1}(\rho(L))$, with $\rho(L)\leq K[t,t^{-1}]$, and $L$ is of type {\rm(iii)}.
\end{proof}

\begin{corollary}\label{cor:char0}
Assume $\ch K=0$. Then the Lie ideals of $\T^-$ not containing $F$ are
$$
0,\qquad K1,\qquad \mathfrak s,\qquad \mathfrak s+K(1+\lambda w)\quad\text{with}\quad\lambda\in K.
$$
The Lie ideals containing $F$ are precisely the spaces $\rho^{-1}(W)$, where $W\leq K[t,t^{-1}]$.
\end{corollary}

\begin{corollary}\label{cor:charpos}
Assume $\ch K=\ell>0$, and put $P_K=\Span\{1,e^{m\ell},(e^*)^{m\ell}:m\geq1\}$. Then the Lie ideals of $\T^-$ not containing $F$ are $0$, $K1$, and the spaces
$$
\mathfrak s+U,\qquad U\leq P_K\oplus Kw\quad\text{with}\quad w\notin U.
$$
The Lie ideals containing $F$ are precisely the spaces $\rho^{-1}(W)$, where $W\leq K[t,t^{-1}]$.
\end{corollary}

We record the inclusion relations determined by
Theorem~\ref{thm:classification}. For a subspace
$U\leq P_K\oplus Kw$ with $w\notin U$, put $L(U)=\mathfrak s\oplus U$, and, for a subspace $W\leq K[t,t^{-1}]$, put $J(W)=\rho^{-1}(W)$.
Then
\begin{align*}
L(U_1)\subseteq L(U_2)
&\quad\text{if and only if}\quad U_1\subseteq U_2,\\
J(W_1)\subseteq J(W_2)
&\quad\text{if and only if}\quad W_1\subseteq W_2,\\
L(U)\subseteq J(W)
&\quad\text{if and only if}\quad \rho(U)\subseteq W,
\end{align*}
and
$$
L(U)\cap F=\mathfrak s,\qquad
L(U)+F=J(\rho(U)).
$$
Moreover,
$$
K1\subseteq L(U)
\quad\text{if and only if}\quad
1\in U,
\qquad
K1\subseteq J(W)
\quad\text{if and only if}\quad
1\in W.
$$
The first two equivalences follow by passing to the quotients by
$\mathfrak s$ and $F$, respectively. The third follows from
$\mathfrak s\subseteq F$ and $\rho(L(U))=\rho(U)$. Since $F\cap(P_K\oplus Kw)=Kw$ and $w\notin U$, one has $U\cap F=0$, which proves the first identity.
The second follows from $\ker\rho=F$, since
$U+F=\rho^{-1}(\rho(U))$. The equivalence involving $L(U)$ follows
from the injectivity of $\pi$ on $P_K\oplus Kw$, while the one
involving $J(W)$ follows from $\rho(1)=1$.

Thus $F=J(0)$ and $L(U)$ are incomparable whenever $U\neq0$. If
$\operatorname{char}K=0$, the interval from $\mathfrak s$ to
$J(K)=\rho^{-1}(K)$ is the subspace lattice of the two-dimensional
space $J(K)/\mathfrak s$. The ideals corresponding to the
one-dimensional subspaces are
$$
F=\mathfrak s\oplus Kw
\quad\text{and}\quad
\mathfrak s\oplus K(1+\lambda w),
\qquad \lambda\in K,
$$
and they are pairwise incomparable. If $\operatorname{char}K=\ell>0$,
the preceding inclusion formulas remain valid, with $U$ ranging over
the subspaces of the larger space $P_K\oplus Kw$ that do not contain
$w$.

\begin{corollary}[Graded Lie ideals]\label{cor:graded-lie-ideals}
Give $\T$ its natural $\mathbb Z$-grading. For $n\in\mathbb Z$, put $F_n=\Span\{F_{ij}:i-j=n\}$. Then
$$
{\T}_n=K e^n\oplus F_n\quad\text{if }n>0,\qquad
{\T}_0=K1\oplus F_0,\qquad
{\T}_n=K(e^*)^{-n}\oplus F_n\quad\text{if }n<0.
$$
Moreover, $F=\bigoplus_{n\in\mathbb Z}F_n$, and $\mathfrak s$ is graded, with $\mathfrak s_n=F_n$ for $n\neq0$ and $\mathfrak s_0=\left\{\sum_i a_iF_{ii}:\sum_i a_i=0\right\}$. The graded Lie ideals of $\T^-$ are precisely $0$, $K1$, the spaces $\mathfrak s+U$, where $U$ is a graded subspace of $P_K\oplus Kw$ and $w\notin U$, and the spaces $\rho^{-1}(W)$, where $W$ is a graded subspace of $K[t,t^{-1}]$.
\end{corollary}

\begin{proof}
By Lemma~\ref{lem:path-basis},
$$
\mathcal B=\{v\}\cup\{e^r:r\geq1\}\cup\{(e^*)^r:r\geq1\}\cup\{F_{ij}:i,j\geq1\}
$$
is a $K$-basis of $\T$. Each element of $\mathcal B$ is homogeneous. More precisely,
$$
\deg v=0,\qquad \deg e^r=r,\qquad\text{and}\qquad \deg (e^*)^r=-r.
$$
Recall that $\alpha_1=w$, $\alpha_i=e^{i-2}f$ for $i\geq2$, and $F_{ij}=\alpha_i\alpha_j^*$. Since $\deg w=0$ and
$$
\deg(e^{i-2}f)=i-1\qquad\text{for }i\geq2,
$$
one has $\deg\alpha_i=i-1$ for all $i\geq1$. It follows that
$$
\deg F_{ij}=\deg\alpha_i-\deg\alpha_j=(i-1)-(j-1)=i-j.
$$

Fix $n>0$. The elements of $\mathcal B$ of degree $n$ are $e^n$ and the matrix units $F_{ij}$ satisfying $i-j=n$. Hence ${\T}_n=Ke^n\oplus F_n$. The basis elements of degree $0$ are $v$ and the diagonal matrix units $F_{ii}$. Since $w=F_{11}\in F_0$ and $v=1-w$, one has $Kv+F_0=K1+F_0$. This sum is direct because $\rho(1)=1$ and $\rho(F_0)=0$ by Lemma~\ref{lem:quotient}. Thus ${\T}_0=K1\oplus F_0$. Finally, if $n<0$, the elements of $\mathcal B$ of degree $n$ are $(e^*)^{-n}$ and the matrix units $F_{ij}$ satisfying $i-j=n$. Therefore ${\T}_n=K(e^*)^{-n}\oplus F_n$.

Since the $F_{ij}$ form a $K$-basis of $F$ by Lemma~\ref{lem:matrix-units}, and each $F_{ij}$ belongs to exactly one $F_n$, we obtain $F=\bigoplus_{n\in\mathbb Z}F_n$. Put $\mathfrak s_n=\mathfrak s\cap{\T}_n$. By Lemma~\ref{lem:s-trace-zero},
$$
\mathfrak s=\operatorname{span}_K\{F_{ij}:i\neq j\}
+\operatorname{span}_K\{F_{ii}-F_{jj}:i,j\geq1\}.
$$
If $n\neq0$ and $i-j=n$, then $i\neq j$, so $F_{ij}\in\mathfrak s$. Hence $F_n\subseteq\mathfrak s_n$. Conversely, $\mathfrak s\subseteq F$ gives $\mathfrak s_n\subseteq F_n$, and therefore $\mathfrak s_n=F_n$ for $n\neq0$.
Since $F_0=\bigoplus_{i\geq1}KF_{ii}$, it follows from  Lemma~\ref{lem:s-trace-zero} that
$$
\mathfrak s_0
=\ker\bigl(\operatorname{Tr}|_{F_0}\bigr)
=\Big\{\sum_i a_iF_{ii}:\text{only finitely many }a_i\neq0
\text{ and }\sum_i a_i=0\Big\}.
$$
Thus $\mathfrak s=\mathfrak s_0\oplus\bigoplus_{n\neq0}F_n$, so $\mathfrak s$ is graded.

We now determine which Lie ideals in Theorem~\ref{thm:classification} are graded. The spaces $0$ and $K1$ are graded. The space $P_K\oplus Kw$ is also graded. If $\operatorname{char}K=0$, then
$$
P_K\oplus Kw=K1\oplus Kw\subseteq{\T}_0.
$$
If $\operatorname{char}K=\ell>0$, then
$$
P_K\oplus Kw
=(K1\oplus Kw)
\oplus\bigoplus_{m\geq1}Ke^{m\ell}
\oplus\bigoplus_{m\geq1}K(e^*)^{m\ell},
$$
where the displayed summands have degrees $0$, $m\ell$, and $-m\ell$, respectively.

Let $U\leq P_K\oplus Kw$ with $w\notin U$. If $U$ is graded, then $\mathfrak s+U$ is graded because both $\mathfrak s$ and $U$ are graded. Conversely, suppose that $\mathfrak s+U$ is graded. The quotient map $\pi:\T\to\T/\mathfrak s$ is graded because $\mathfrak s$ is graded. Moreover, as established immediately before Theorem~\ref{thm:classification}, the restriction
$$
\pi|_{P_K\oplus Kw}:P_K\oplus Kw\longrightarrow Z(\T/\mathfrak s)
$$
is injective. Hence
$$
\mathfrak s\cap(P_K\oplus Kw)=0
\qquad\text{and}\qquad
(\mathfrak s+U)\cap(P_K\oplus Kw)=U.
$$
Let $u\in U$, and write its homogeneous decomposition in the graded space $P_K\oplus Kw$ as
$$
u=\sum_{n\in\mathbb Z}u_n,\qquad\text{where } u_n\in(P_K\oplus Kw)\cap{\T}_n.
$$
Since $u\in\mathfrak s+U$ and $\mathfrak s+U$ is graded, every $u_n$ belongs to $\mathfrak s+U$. The preceding intersection identity then gives $u_n\in U$ for every $n$. Thus $U$ is graded. We have proved that
$$
\mathfrak s+U\text{ is graded}
\quad\text{if and only if}\quad
U\text{ is graded}.
$$

Next let $W\leq K[t,t^{-1}]$. By Lemma~\ref{lem:quotient} and the formulas for the homogeneous components of $\T$ proved above,
$$
\rho({\T}_n)\subseteq Kt^n\qquad\text{for all }n\in\mathbb Z.
$$
Thus $\rho$ is a graded homomorphism, where $K[t,t^{-1}]$ has its standard grading $K[t,t^{-1}]=\bigoplus_{n\in\mathbb Z}Kt^n$. If $W$ is graded, then $\rho^{-1}(W)$ is graded. Indeed, let $x\in\rho^{-1}(W)$ and write
$$
x=\sum_n x_n,\qquad\text{where } x_n\in{\T}_n.
$$
The elements $\rho(x_n)\in Kt^n$ are the homogeneous components of $\rho(x)\in W$. Since $W$ is graded, $\rho(x_n)\in W$ for every $n$, and hence $x_n\in\rho^{-1}(W)$.

Conversely, suppose that $\rho^{-1}(W)$ is graded. Let $f\in W$. Since $\rho$ is surjective, choose $x\in\rho^{-1}(W)$ with $\rho(x)=f$, and write $x=\sum_nx_n$ with $x_n\in{\T}_n$. The gradedness of $\rho^{-1}(W)$ gives $x_n\in\rho^{-1}(W)$ for every $n$. Hence $\rho(x_n)\in W$, and these are precisely the homogeneous components of $f$. Therefore $W$ is graded. We have proved that
$$
\rho^{-1}(W)\text{ is graded}
\quad\text{if and only if}\quad
W\text{ is graded}.
$$

Theorem~\ref{thm:classification} now gives the asserted list of graded Lie ideals. Finally, since each homogeneous component $Kt^n$ of $K[t,t^{-1}]$ is one-dimensional, a subspace $W\leq K[t,t^{-1}]$ is graded if and only if
$$
W=\bigoplus_{n\in S}Kt^n
=\operatorname{span}_K\{t^n:n\in S\}
$$
for some subset $S\subseteq\mathbb Z$. Similarly, if $\operatorname{char}K=\ell>0$, a graded subspace of $P_K\oplus Kw$ has the form
$$
U=U_0
\oplus\bigoplus_{m\in S_+}Ke^{m\ell}
\oplus\bigoplus_{m\in S_-}K(e^*)^{m\ell},
$$
where $U_0\leq K1\oplus Kw$, $S_+,S_-\subseteq\mathbb N$, and the condition $w\notin U$ is equivalent to $w\notin U_0$.
\end{proof}

\begin{remark}
It is known that graded associative ideals of $\T$ are only $0$, $F$, and $\T$. Thus the graded Lie ideal lattice is already much larger than the graded associative ideal lattice.
\end{remark}

\section{Derivations}

The associative derivations of the Jacobson algebra were described by Gerritzen, and Lopatkin later described derivations and outer derivations of Leavitt path algebras, including the outer derivation algebra of the Toeplitz algebra; see \cite{Gerritzen2000,Lopatkin2019}. We recall only the form needed below, and then pass from associative derivations of $\T$ to Lie derivations of $\T^-$. Thus the associative part is not new; what is recorded here is the passage from associative derivations to Lie derivations of $\T^-$.

A $K$-linear map $D:\T\to\T$ is an associative derivation if $D(ab)=D(a)b+aD(b)$ for all $a,b\in\T$. A $K$-linear map $\Delta:\T\to\T$ is a Lie derivation of $\T^-$ if
$$
\Delta([a,b])=[\Delta(a),b]+[a,\Delta(b)]
\qquad\text{where }a,b\in\T.
$$
Every associative derivation is a Lie derivation. Since $Z(\T)=K1$ and $[\T,\T]=F$, the central Lie derivations are precisely the maps
$$
C_\lambda(a)=\lambda(\rho(a))1,
\qquad \text{where }\lambda\in\operatorname{Hom}_K(K[t,t^{-1}],K).
$$

For this section set $X=e^*+f^*$ and $Y=e+f$. Then $XY=1$ and $YX=v$. For
$$
r=a_0+\sum_{n\geq1}a_nt^n+\sum_{n\geq1}a_{-n}t^{-n}\in K[t,t^{-1}],
$$
put
$$
\widetilde{r}=a_0 1+\sum_{n\geq1}a_ne^n+\sum_{n\geq1}a_{-n}(e^*)^n.
$$
Define $\partial_r$ on the generators $X,Y$ by
$$
\partial_r(Y)=\widetilde{r}\quad\text{and}\quad \partial_r(X)=-X\widetilde{r}X.
$$
The defining relation is preserved, since
$$
\partial_r(X)Y+X\partial_r(Y)=(-X\widetilde{r}X)Y+X\widetilde{r}=0.
$$
Thus $\partial_r$ is a well-defined associative derivation of $\T$. Modulo $F$, it induces the derivation of $K[t,t^{-1}]$ sending $t$ to $r$.

\begin{proposition}\label{prop:assoc-derivations-toeplitz}
Every associative $K$-derivation of $\T$ has the form $D=\operatorname{ad}a+\partial_r$, with $a\in\T$ and $r\in K[t,t^{-1}]$. The element $r$ is uniquely determined by $D$, and $a$ is unique modulo $K1$.
\end{proposition}

\begin{proof}
This is Gerritzen's description of the associative derivations of the Jacobson algebra given in \cite[Proposition~7.4 and Corollary~7.5]{Gerritzen2000}, written in the present notation. We recall the argument. Since $w^2=w$, one has $D(w)=D(w)w+wD(w)\in F$. As $F=\T w\T$, it follows that $D(F)\subseteq F$. Hence $D$ induces a derivation of $\T/F\cong K[t,t^{-1}]$.

Let $r=\rho(D(Y))$, and replace $D$ by $D-\partial_r$. Then $D(Y),D(X)\in F$. Write $D(Y)=\sum a_{ij}F_{ij}$. Since $[Y^{i-1}X^j,Y]=F_{ij}$, the element $b_1=\sum a_{ij}Y^{i-1}X^j$ satisfies $[b_1,Y]=D(Y)$. Replacing $D$ by $D-\operatorname{ad}b_1$, we may assume $D(Y)=0$.

Applying $D$ to $XY=1$ gives $D(X)Y=0$. Write $D(X)=\sum b_{ij}F_{ij}$. Since $F_{ij}Y=F_{i,j-1}$ if $j>1$, and $F_{i1}Y=0$, we get $D(X)=\sum_i\lambda_iF_{i1}$. Since $[Y^i,X]=-F_{i1}$, the element $b_2=-\sum_i\lambda_iY^i$ satisfies $[b_2,Y]=0$ and $[b_2,X]=D(X)$. Thus, after subtracting $\partial_r$, the derivation is inner.

Finally, suppose $\operatorname{ad}a+\partial_r=0$. Passing to $\T/F\cong K[t,t^{-1}]$, the inner part vanishes, while $\partial_r$ sends $t$ to $r$. Hence $r=0$. Then $\operatorname{ad}a=0$, so $a\in Z(\T)=K1$. This proves the uniqueness assertions.
\end{proof}

\begin{lemma}\label{lem:centralizer-F-rcf}
The centralizer of $F$ in $M_{rcf}(K)$ is $K1$.
\end{lemma}

\begin{proof}
Let $B=(b_{pq})\in M_{rcf}(K)$ commute with every $F_{ij}$. From $[B,F_{ii}]=0$ we get $b_{pi}=0$ for $p\neq i$ and $b_{iq}=0$ for $q\neq i$. Since this holds for every $i$, the matrix $B$ is diagonal. Write $B=\operatorname{diag}(b_1,b_2,\ldots)$. Then
$$
[B,F_{ij}]=(b_i-b_j)F_{ij}.
$$
Thus $b_i=b_j$ for all $i,j$. Hence $B$ is a scalar multiple of the infinite identity matrix.
\end{proof}

\begin{lemma}\label{lem:finitary-lie-derivations}
Let $\delta:F^-\to F^-$ be a Lie derivation. Then there is $B\in M_{rcf}(K)$ such that $\delta(f)=[B,f]$ for all $f\in F$. Moreover, $B$ is unique modulo scalar infinite matrices.
\end{lemma}

\begin{proof}
Put $p_i=F_{ii}$. Write $\delta(p_k)=\sum h^{(k)}_{ab}F_{ab}$, a finite sum. If $a\neq b$ and $a,b\neq k$, then applying $\delta$ to $[p_a,p_k]=0$ and comparing the $(a,b)$-entry gives $h^{(k)}_{ab}=0$. Thus the off-diagonal part of $\delta(p_k)$ lies in the $k$-th row and the $k$-th column. For $i\neq j$, applying $\delta$ to $[p_i,p_j]=0$ and comparing the $(i,j)$-entry yields $h^{(i)}_{ij}+h^{(j)}_{ij}=0$.

Define $A=(a_{ij})$ by $a_{ii}=0$ and $a_{ij}=-h^{(i)}_{ij}$ for $i\neq j$. Then $A\in M_{rcf}(K)$: row-finiteness follows from the finite sum $\delta(p_i)$, and column-finiteness follows from $a_{ij}=h^{(j)}_{ij}$. Set $\delta_1=\delta-\operatorname{ad}A$. By construction, each $\delta_1(p_i)$ is a finite diagonal matrix.

We first show that $\delta_1(p_i)=0$. Write $\delta_1(p_i)=\sum_r d_rF_{rr}$. For $j\neq i$, applying $\delta_1$ to $[p_i,F_{ij}]=F_{ij}$ and comparing the $(i,j)$-entry gives $d_i=d_j$. Since $\delta_1(p_i)$ is finite, choose $j\neq i$ with $d_j=0$. Hence all $d_r$ are zero.

Now fix $i\neq j$, and put $Z_{ij}=\delta_1(F_{ij})$. From
$$
[p_k,F_{ij}]=0\quad(k\notin\{i,j\}),\qquad
[p_i,F_{ij}]=F_{ij},\qquad
[p_j,F_{ij}]=-F_{ij},
$$
and from $\delta_1(p_k)=0$, it follows that $Z_{ij}$ has the form
$$
Z_{ij}=x_{ij}F_{ij}+y_{ij}F_{ji}.
$$
Choose $k$ distinct from $i,j$. Applying $\delta_1$ to $[F_{ij},F_{jk}]=F_{ik}$ gives
$$
x_{ik}F_{ik}+y_{ik}F_{ki}=(x_{ij}+x_{jk})F_{ik}.
$$
Thus $y_{ik}=0$. Since every ordered pair of distinct indices can be written as $(i,k)$ with a third index $j$, all $y_{ij}$ vanish. Hence $\delta_1(F_{ij})=x_{ij}F_{ij}$ for $i\neq j$.

Applying $\delta_1$ to $[F_{ij},F_{ji}]=p_i-p_j$ we get $x_{ij}+x_{ji}=0$. Applying it to $[F_{ij},F_{jk}]=F_{ik}$ gives $x_{ik}=x_{ij}+x_{jk}$ for distinct $i,j,k$. Put $h_1=0$ and $h_i=x_{i1}$ for $i>1$, and let $H=\operatorname{diag}(h_1,h_2,\ldots)$. Then $x_{ij}=h_i-h_j$, so $[H,F_{ij}]=x_{ij}F_{ij}$ and $[H,p_i]=0$. Therefore $\delta-\operatorname{ad}(A+H)$ kills every matrix unit, and hence is zero on $F$. Taking $B=A+H$ proves existence.

If both $B$ and $B'$ implement $\delta$, then $B-B'$ centralizes $F$. By Lemma~\ref{lem:centralizer-F-rcf}, $B-B'\in K1$. This proves uniqueness modulo scalar infinite matrices.
\end{proof}

We write $\operatorname{Der}_{\mathrm{Lie}}(\T^-)$ for the vector space of Lie derivations of $\T^-$, and $\operatorname{Inn}(\T)=\{\operatorname{ad}a:a\in\T\}$.

\begin{theorem}\label{thm:lie-derivations-toeplitz}
Every Lie derivation of $\T^-$ has the form
$$
\Delta=\operatorname{ad}a+\partial_r+C_\lambda,
\qquad \text{where }a\in\T, r\in K[t,t^{-1}],
\lambda\in\operatorname{Hom}_K(K[t,t^{-1}],K).
$$
Moreover, we have the vector-space decomposition
$$
\operatorname{Der}_{\mathrm{Lie}}(\T^-)
=
\operatorname{Inn}(\T)
\oplus\{\partial_r:r\in K[t,t^{-1}]\}
\oplus\{C_\lambda:\lambda\in\operatorname{Hom}_K(K[t,t^{-1}],K)\}.
$$
Equivalently, as vector spaces,
$$
\operatorname{Der}_{\mathrm{Lie}}(\T^-)
\cong
\T/K1\oplus K[t,t^{-1}]\frac{d}{dt}
\oplus \operatorname{Hom}_K(K[t,t^{-1}],K).
$$
\end{theorem}

\begin{proof}
For every $x,y\in\T$, one has
$$
\Delta([x,y])=[\Delta(x),y]+[x,\Delta(y)]\in[\T,\T].
$$
Since $F=[\T,\T]$, the linearity therefore yields $\Delta(F)\subseteq F$, and $\Delta|_F$ is a Lie derivation of $F^-$.

Before applying Lemma~\ref{lem:finitary-lie-derivations}, we realize $\T$ and $F$ inside a common matrix algebra. Put $\mathscr V=\T w$. Then $\mathscr V$ has basis $\{\alpha_i:i\geq1\}$. Left multiplication defines a representation
$$
\Phi:\T\longrightarrow\operatorname{End}_K(\mathscr V).
$$
The matrix-unit relations and the definitions of $X$ and $Y$ implies that $F_{ij}\alpha_k=\delta_{jk}\alpha_i$, together with $Y\alpha_i=\alpha_{i+1}$, $X\alpha_1=0$, and $X\alpha_i=\alpha_{i-1}$ for $i>1$.
Thus, relative to the basis $\{\alpha_i:i\geq1\}$, we have
$$
\Phi(F_{ij})=E_{ij},
\qquad
\Phi(Y)=\sum_{i\geq1}E_{i+1,i},
\qquad\text{and}\qquad
\Phi(X)=\sum_{i\geq1}E_{i,i+1}.
$$
In particular, the image of $\Phi$ is contained in $\mathrm M_{rcf}(K)$.

We next verify that $\Phi$ is faithful. By Lemma~\ref{lem:path-basis}, every $x\in\T$ has a unique expression
$$
x=\lambda_0 1+\sum_{n=1}^{N}\lambda_ne^n
+\sum_{n=1}^{M}\mu_n(e^*)^n+f_0,
\text{where }f_0\in F.
$$
Suppose that $\Phi(x)=0$. Choose $j>M+1$ larger than every column index occurring in the finitary matrix $f_0$. Then $f_0\alpha_j=0$, and direct calculation shows
$$
x\alpha_j
=\lambda_0\alpha_j
+\sum_{n=1}^{N}\lambda_n\alpha_{j+n}
+\sum_{n=1}^{M}\mu_n\alpha_{j-n}.
$$
The basis vectors in this expression are distinct, so $\lambda_0=\lambda_n=\mu_n=0$ for all $n$. It follows that $x=f_0$. Since $F_{ij}$ act as the standard matrix units on $\mathscr V$, the action of $F$ on $\mathscr V$ is faithful, and hence $f_0=0$. Therefore $\Phi$ is injective. We may consequently identify $\T\subseteq\mathrm M_{rcf}(K)$ and $F=\mathrm M_\infty(K)$.

By Lemma~\ref{lem:finitary-lie-derivations}, there is $B\in\mathrm M_{rcf}(K)$ such that $\Delta(f)=[B,f]$, where $f\in F$. We claim that $[B,Y]$ and $[B,X]$ belong to $\T$. Let $f\in F$. Since $[Y,f]\in F$, the preceding identity implies
$$
\Delta([Y,f])=[B,[Y,f]].
$$
On the other hand, the Lie derivation identity yields
$$
\Delta([Y,f])=[\Delta(Y),f]+[Y,[B,f]],
$$
whereas the Jacobi identity gives
$$
[B,[Y,f]]=[[B,Y],f]+[Y,[B,f]].
$$
Comparing these expressions, we obtain
$$
[[B,Y]-\Delta(Y),f]=0,
\qquad\text{where } f\in F.
$$
Thus $[B,Y]-\Delta(Y)$ centralizes $F$ in $\mathrm M_{rcf}(K)$. By Lemma~\ref{lem:centralizer-F-rcf}, we have $[B,Y]-\Delta(Y)\in K1$. Since $\Delta(Y)\in\T$, it follows that $[B,Y]\in\T$. The same argument, with $X$ in place of $Y$, shows that $[B,X]\in\T$.

The elements $X$ and $Y$ generate $\T$, and $\operatorname{ad}B$ is an associative derivation of $\mathrm M_{rcf}(K)$. Since $[B,X]$ and $[B,Y]$ belong to $\T$, repeated use of $[B,xy]=[B,x]y+x[B,y]$ we get  that $[B,\T]\subseteq\T$. Hence
$$
D_B:\T\longrightarrow\T,
\qquad\text{defined by }
D_B(x)=[B,x],
$$
is an associative derivation of $\T$. Moreover, we have $D_B(f)=[B,f]=\Delta(f)$, where $f\in F$. Therefore, Proposition~\ref{prop:assoc-derivations-toeplitz} now applied to yield $D_B=\operatorname{ad}a+\partial_r$ for some $a\in\T$ and $r\in K[t,t^{-1}]$.

Set $\Gamma=\Delta-D_B$. Since every associative derivation is a Lie derivation, $\Gamma$ is a Lie derivation of $\T^-$. The preceding equality on $F$ shows that $\Gamma(F)=0$. If $x\in\T$ and $f\in F$, then $[x,f]\in F$, and hence
$$
0=\Gamma([x,f])
=[\Gamma(x),f]+[x,\Gamma(f)]
=[\Gamma(x),f].
$$
Thus $\Gamma(x)$ centralizes $F$ for every $x\in\T$. It follows form Lemma~\ref{lem:centralizer-F-rcf} that $\Gamma(x)\in K1$ where $\qquad x\in\T$. Moreover, it is clear that $\Gamma$ vanishes on $F=[\T,\T]$. Consequently, $\Gamma$ factors uniquely through $\rho:\T\to K[t,t^{-1}]$, and there is a unique $\lambda\in\operatorname{Hom}_K(K[t,t^{-1}],K)$ such that
$$
\Gamma(x)=\lambda(\rho(x))1=C_\lambda(x),
\qquad \text{where }x\in\T.
$$
This provides the existence of the stated decomposition.

It remains to prove directness. Suppose that $\operatorname{ad}a+\partial_r+C_\lambda=0$. Passing to $\T/F\cong K[t,t^{-1}]$ eliminates the inner summand and yields $d_r+c_\lambda=0$, where $d_r$ is the derivation determined by $d_r(t)=r$, and $c_\lambda(s)=\lambda(s)$. Evaluating this identity at $t$ we get $r+\lambda(t)=0$, so $r=\gamma\in K$, where $\gamma=-\lambda(t)$. Suppose that $\gamma\neq0$. If $\operatorname{char}K=0$, take $N=2$; if $\operatorname{char}K=p>0$, take $N=p+1$. In either case, $N>1$ and $N\cdot1_K\neq0$. Evaluating at $t^N$ we have
$$
0=d_r(t^N)+c_\lambda(t^N)
=N\gamma t^{N-1}+\lambda(t^N)1.
$$
Since $N\gamma\neq0$ and $t^{N-1}$ is linearly independent from $1$ in $K[t,t^{-1}]$, this is impossible. Hence $\gamma=0$, and therefore $r=0$. It follows that $c_\lambda=0$, so $\lambda=0$. Finally, $\operatorname{ad}a=0$, whence $a\in Z(\T)=K$. Thus the three summands have pairwise trivial intersection, and the decomposition is direct.
\end{proof}

We use the natural $\mathbb Z$-grading $\T=\bigoplus_{n\in\mathbb Z}{\T}_n$. A derivation $D$ of $\T$ is homogeneous of degree $d$ if $D({\T}_n)\subseteq{\T}_{n+d}$ for all $n$. The same definition applies to Lie derivations of $\T^-$. Degree-zero homogeneous derivations are called graded derivations. For $n\in\mathbb Z$, let $\varepsilon_n:K[t,t^{-1}]\to K$ be the coefficient functional defined by $\varepsilon_n(t^m)=\delta_{n,m}$.

\begin{corollary}\label{cor:homogeneous-derivations}
Let $d\in\mathbb Z$. The homogeneous associative derivations of $\T$ of degree $d$ are precisely
$$
\operatorname{ad}a+\lambda\partial_{t^{d+1}},
\qquad\text{where } a\in{\T}_d \text{ and }\lambda\in K.
$$
The homogeneous Lie derivations of $\T^-$ of degree $d$ are precisely
$$
\operatorname{ad}a+\lambda\partial_{t^{d+1}}+\mu C_{\varepsilon_{-d}},
\qquad\text{where }a\in{\T}_d\text{ and }\lambda,\mu\in K.
$$
\end{corollary}

\begin{proof}
We first record the degrees of the three types of derivations appearing in Theorem~\ref{thm:lie-derivations-toeplitz}. If $a\in{\T}_k$, then
$$
[{\T}_n,a]\subseteq{\T}_{n+k},
\qquad n\in\mathbb Z,
$$
so $\operatorname{ad}a$ is homogeneous of degree $k$.

For $m\in\mathbb Z$, the element $\widetilde{t^m}$ is homogeneous of degree $m$. Since $\deg(Y)=1$ and $\deg(X)=-1$, one has
$$
\partial_{t^m}(Y)=\widetilde{t^m}\in{\T}_m
\quad\text{and}\quad
\partial_{t^m}(X)=-X\widetilde{t^m}X\in{\T}_{m-2}.
$$
Thus $\partial_{t^m}$ shifts the degrees of both generators by $m-1$. Since $X$ and $Y$ generate $\T$, the derivation $\partial_{t^m}$ is homogeneous of degree $m-1$.

Finally, $C_{\varepsilon_n}$ vanishes on every homogeneous component of $\T$ except ${\T}_n$ modulo $F$, and its image is contained in $K1\subseteq{\T}_0$. Hence $C_{\varepsilon_n}$ has degree $-n$. In particular, $C_{\varepsilon_{-d}}$ has degree $d$.

Let $D$ be a homogeneous associative derivation of degree $d$. By Proposition~\ref{prop:assoc-derivations-toeplitz}, write
$$
D=\operatorname{ad}a+\partial_r,
\qquad\text{where } a\in\T, r\in K[t,t^{-1}].
$$
Decompose $a$ and $r$ into their homogeneous components as
$$
a=\sum_{k\in\mathbb Z}a_k
\quad\text{and}\quad
r=\sum_{m\in\mathbb Z}r_mt^m,
$$
where $a_k\in{\T}_k$, $r_m\in K$, and both sums are finite. Then
$$
D=\sum_{k\in\mathbb Z}
\left(
\operatorname{ad}a_k+
r_{k+1}\partial_{t^{k+1}}
\right),
$$
and the summand indexed by $k$ has degree $k$.

Since $D$ has degree $d$, every summand of degree $k\neq d$ is zero. The uniqueness assertion in Proposition~\ref{prop:assoc-derivations-toeplitz} then implies
$$
r_{k+1}=0
\quad\text{and}\quad
a_k\in Z(\T)=K1,
\qquad k\neq d.
$$
Since $K1\subseteq{\T}_0$, one has $a_k=0$ whenever $k\neq0$, while a scalar component $a_0$ has zero inner derivation and may be discarded. Consequently,
$$
D=\operatorname{ad}a_d+r_{d+1}\partial_{t^{d+1}}.
$$
Writing $\lambda=r_{d+1}$ proves the associative assertion. Conversely, both summands in the displayed expression have degree $d$, so every such derivation is homogeneous of degree $d$.

Now let $\Delta$ be a homogeneous Lie derivation of degree $d$. By Theorem~\ref{thm:lie-derivations-toeplitz}, write
$$
\Delta=\operatorname{ad}a+\partial_r+C_\eta,
\qquad
a\in\T,\quad
r\in K[t,t^{-1}],\quad
\eta\in\operatorname{Hom}_K(K[t,t^{-1}],K).
$$
Retain the homogeneous decompositions of $a$ and $r$ above. The central summand admits the pointwise decomposition
$$
C_\eta
=
\sum_{k\in\mathbb Z}
\eta(t^{-k})C_{\varepsilon_{-k}}.
$$
This sum is well defined because $\rho(x)$ has finite Laurent support for every $x\in\T$. Therefore
$$
\Delta
=
\sum_{k\in\mathbb Z}
\left(
\operatorname{ad}a_k
+r_{k+1}\partial_{t^{k+1}}
+\eta(t^{-k})C_{\varepsilon_{-k}}
\right),
$$
where the summand indexed by $k$ has degree $k$.

For $k\neq d$, homogeneity of $\Delta$ forces the corresponding summand to vanish. The directness in Theorem~\ref{thm:lie-derivations-toeplitz} then yields
$$
r_{k+1}=0,
\qquad
\eta(t^{-k})=0,
\qquad\text{and}\qquad
a_k\in K.
$$
As in the associative case, the scalar components of $a$ may be omitted. Only the component of degree $d$ remains, and hence
$$
\Delta
=
\operatorname{ad}a_d
+r_{d+1}\partial_{t^{d+1}}
+\eta(t^{-d})C_{\varepsilon_{-d}}.
$$
Taking $\lambda=r_{d+1}$ and $\mu=\eta(t^{-d})$ proves the Lie assertion. Conversely, each of the three summands in the displayed expression has degree $d$, so every such Lie derivation is homogeneous of degree $d$.
\end{proof}
\section*{Declarations}

\noindent\textbf{Funding.}
The second author (H.V. Khanh) was supported by the Vietnam National Foundation for Science and
Technology Development (NAFOSTED) under Grant No.~101.04-2025.41.

\medskip
\noindent\textbf{Competing interests.}
The author has no relevant financial or non-financial interests to disclose.

\medskip
\noindent\textbf{Data availability.}
No data were generated or analysed in this study.

\end{document}